\documentclass{amsart}
\usepackage[margin=3cm]{geometry}
\usepackage{lipsum}
\usepackage{xcolor}
\usepackage{hyperref}
\usepackage{bbm}
\usepackage{graphicx}
\usepackage{subcaption}
\usepackage{enumitem}

\newtheorem{theorem}{Theorem}[section]
\newtheorem{prop}[theorem]{Proposition}

\theoremstyle{definition}
\newtheorem{definition}[theorem]{Definition}
\newtheorem{example}[theorem]{Example}

\theoremstyle{remark}

\numberwithin{equation}{section}

\newcommand{\QED}{\hspace*{\fill}\rule{7pt}{7pt}\smallskip}

\newcommand{\vep}{\varepsilon}

\newcommand{\sigmavec}  {\boldsymbol{\sigma}}
\newcommand{\tauvec}  {\boldsymbol{\tau}}

\newcommand{\indic}{\mathbbm{1}}

\begin{document}

\title{Mathematical aspects of the Digital Annealer's simulated annealing algorithm}

\author{Bruno Hideki Fukushima-Kimura}
\address{Faculty of Science, Hokkaido University, Kita 10, Nishi 8, Kita-ku, Sapporo, Hokkaido 060-0810, Japan}

\email{bruno@math.sci.hokudai.ac.jp}

\author{Noe Kawamoto}
\address{Graduate School of Science, Hokkaido University, Kita 10, Nishi 8, Kita-ku, Sapporo, Hokkaido 060-0810, Japan}
\email{kawamoto.noe.d1@elms.hokudai.ac.jp}

\author{Eitaro Noda}
\address{Graduate School of Science, Hokkaido University, Kita 10, Nishi 8, Kita-ku, Sapporo, Hokkaido 060-0810, Japan}
\email{eitaronodax@eis.hokudai.ac.jp}

\author{Akira Sakai}
\address{Faculty of Science, Hokkaido University, Kita 10, Nishi 8, Kita-ku, Sapporo, Hokkaido 060-0810, Japan}
\email{sakai@math.sci.hokudai.ac.jp}

\begin{abstract}
The Digital Annealer is a CMOS hardware designed by Fujitsu Laboratories for high-speed solving of Quadratic Unconstrained Binary Optimization (QUBO) problems that could be difficult to solve by means of existing general-purpose computers. In this paper, we present a mathematical description of the first-generation Digital Annealer's Algorithm from the Markov chain theory perspective, establish a relationship between its stationary distribution with the Gibbs-Boltzmann distribution, and provide a necessary and sufficient condition on its cooling schedule that ensures asymptotic convergence to the ground states.
\end{abstract}

\maketitle

\section{Introduction}
Problems that involve, in a certain sense, an optimal choice of a configuration or a set of parameters among a large number of possibilities commonly arise in several problems of practical and theoretical interest, especially in various fields related to engineering, computer science, artificial intelligence, machine learning, logistics, and very-large-scale integration (VLSI) \cite{Wong1988}. The study of the so-called combinatorial optimization problems \cite{Lawler1976, Papa82} has been revealed to be relevant primarily because of their recurrent appearance in real-world issues and has been accompanied by a large number of proposals for their solution.

The Metropolis algorithm \cite{Metropolis53} is a widespread Markov chain Monte Carlo (MCMC) method for simulating the evolution of a solid in contact with a heat reservoir to thermal equilibrium. Almost thirty years after its introduction, Kirkpatrick et al. \cite{SA1983} and Černy \cite{Cerny1985} independently recognized the connection between combinatorial optimization problems and statistical mechanics motivated by their investigations in proposing a Monte Carlo algorithm approach for integrated circuit design problems and the traveling salesman problem. Their approach was essentially based upon the analogy coming from condensed matter physics of the thermal annealing of solids for obtaining low-energy states. 

Kirkpatrick et al. \cite{SA1983} showed that some combinatorial optimization problems of practical interest could be unambiguously translated (see \cite{Lucas2014} for more examples) into the minimization problem of the Hamiltonian of an Ising model in such a way that each optimal solution for the original problem corresponds to a ground state of the Hamiltonian, and vice versa. The simulated annealing algorithm (SA) is a method whose control parameter is regarded as the temperature of a physical system and consists of iterations of the Metropolis algorithm initially evaluated at very high temperatures and subsequently at slowly decreasing values of the temperature in such a way as to guarantee the convergence to a candidate for a minimal energy state. Still, under the analogy with physical systems, it is well-known that if the cooling is too rapid, the system cannot achieve thermal
equilibrium for each temperature value, which may result in a configuration with defects in the form of high-energy, metastable, locally optimal structures. A treatment of this topic from a theoretical point of view based on the Markov chain theory is present in \cite{AK1989,Gelfand1985,Hajek88,Lundy1986,Mitra1985,LA1987}, where specifically in  \cite{Hajek88} was derived a necessary and sufficient condition on the cooling speed that guarantees asymptotic convergence of the SA to the ground states.

Achieving rigorous results in order to obtain a deep understanding of algorithms proposed to solve combinatorial optimization problems is becoming highly desirable, especially due to the increasing demand for methods for efficiently solving large-scale problems. In \cite{Liu2020}, the authors explored the mathematical foundations of dynamical system algorithms to support their applications in coherent Ising machines (CIM) \cite{CIM1,CIM2}, Kerr-nonlinear parametric oscillators \cite{KPO}, and the simulated bifurcation (SB) algorithm \cite{SB} in the search for Hamiltonian minimizers. In \cite{SCA21}, the authors provided a theoretical treatment of a simulated annealing algorithm based on probabilistic cellular automata called Stochastic Cellular Automata Annealing (SCA) \cite{STATICA} motivated by the potential to make use of its parallelizability to provide solutions within a shorter amount of time as compared to SA. The authors showed that an appropriate cooling rate in the form $T_k \propto \frac{1}{\log(k)}$ is sufficient to guarantee the asymptotic convergence to ground states. Experimental results in \cite{SCA21, ISCIE2022, TISCIE2022, IEEE2023} showed that the SCA performed better than SA in searching for ground states in most of the considered scenarios, while its variation $\varepsilon$-SCA (or RPA) surpassed the performance of SCA and SA in all scenarios; however, there is still no mathematical foundation that supports the convergence of the $\varepsilon$-SCA.

The so-called Digital Annealer \cite{DA2019, MatsubaraDA, DA2017, DA3rd, Fujitsu2017}, developed by Fujitsu Laboratories, is a CMOS hardware designed to solve fully connected Quadratic Unconstrained Binary Optimization (QUBO) problems whose performance has been compared with other state-of-the-art dedicated solvers in terms of solution quality (success probabilities) and scalability (time to solution).
For an up-to-date, extensive review of several dedicated hardware solvers for the Ising model, including Fujitsu's Digital Annealer benchmarks against multiple well-known Ising machines, we highly recommend \cite{Mohseni2022}.
The first-generation Digital Annealer's Algorithm \cite{DA2019}, or simply DA, is a physics-inspired algorithm (motivated by SA) employed in the search for ground states of a Hamiltonian by relying on a parallel-trial scheme and an escape mechanism (called a \emph{dynamic offset}) as an attempt to increase spin-flip acceptance probabilities and avoid local minima.
In \cite{DA2019}, the DA was benchmarked against SA, PT (PT + ICM) \cite{Zhu2020, Zhu2015}, and PTDA \cite{DA2019}, considering sparse and fully connected Ising models: two-dimensional (with periodic boundary condition) and Sherrington-Kirkpatrick (SK) spin-glasses with bimodal and Gaussian disorder. For sparse models, the PT + ICM exhibited the lowest time to solution (TTS) with a clear scaling advantage; moreover, for sparse bimodal models, the parallel-trial scheme incorporated by the DA was not enough to decrease the TTS and provide better scaling than SA. On the other hand, the performance for fully connected models was significantly better, where the analysis of TTS for DA against SA, PT, and PTDA revealed a significant and consistent speedup of at least two orders of magnitude, and the DA achieved higher success rates and a small scaling advantage over SA.
In \cite{IEEE2023}, the authors benchmarked the DA against SA, SCA, and $\vep$-SCA, considering extensive series of models with different densities for the spin-spin interactions $J_{i,j} \in \{-1,0,1\}$. The authors obtained higher success probabilities for DA against SA and SCA in all instances, while $\vep$-SCA outperformed DA except in the cases where the non-null interactions were dense and anti-ferromagnetic. In \cite{TISCIE2022}, DA clearly outperformed SA, SCA, and $\vep$-SCA in obtaining solutions for the traveling salesman problem with higher success probability. 

 Although further iterations of the first-generation DA also rely on replica exchange methods \cite{DA3rd}, in this paper, we restrict ourselves to the mathematical description of the DA in the particular scenario where only the parallel search technique is regarded. Even though we restrict ourselves to such a particular case, this simple setting is mathematically interesting from the Markov chain theory point-of-view and has resulted in significant practical improvement over SA and other novel simulated annealing algorithms in several instances \cite{TISCIE2022,IEEE2023, MatsubaraDA, Mohseni2022}. We show that the equilibrium distribution corresponding to the finite fixed temperature Markov chain assumes a non-trivial form and does not necessarily coincide with the Gibbs-Boltzmann distribution. Moreover, based on the techniques developed by Hajek \cite{Hajek88}, we obtain a necessary and sufficient condition on the cooling schedule that guarantees asymptotic convergence of the algorithm to a ground state.

\section{Digital Annealer's Algorithm}\label{sec:DA}

This paper's main problem of interest concerns finding points of global minima, the so-called ground states, of a Hamiltonian function for a given Ising model on a graph. Let $G = (V, E)$ be an undirected graph with a finite vertex set, containing no multi-edges and no self-loops, and let us consider the set $\Omega = \{-1,+1\}^{V}$ of all spin configurations  $\sigmavec = (\sigma_{x})_{x \in V}$ whose spin values are either $+1$ or $-1$. 
Given a family of spin-spin coupling constants  $(J_{x,y})_{x,y\in V}$ which is symmetric (that is, $J_{x,y} = J_{y,x}$ for every pair $x,y$ in $V$) and satisfies
$J_{x,y}=0$ whenever $\{x,y\}\notin E$, and a family of local external fields 
$(h_x)_{x\in V}$, let the Hamiltonian  of the Ising model be the real-valued function $H$  defined on $\Omega$ given by
\begin{align}\label{eq:Hamiltonian}
H(\sigmavec)=-\frac12\sum_{x,y\in V}J_{x,y}\sigma_x\sigma_y-\sum_{x\in V}h_x
 \sigma_x
\end{align}
for each spin configuration $\sigmavec = (\sigma_x)_{x\in V}$. Then, the set of ground states of $H$ is defined by
\begin{equation}
    \text{GS} = \{\sigmavec \in \Omega: H(\sigmavec) = \min_{\tauvec \in \Omega} H(\tauvec)\}. 
\end{equation}

Before we proceed, let us introduce some notation. Given a configuration $\sigmavec$ and a vertex $x$, let us denote by $\sigmavec^{x}$ the spin configuration in $\Omega$ which coincides with $\sigmavec$ except for its spin value at $x$, more precisely, we define
\begin{equation}
(\sigmavec^{x})_{y} = 
\begin{cases}
-\sigma_{x} & \text{if $y = x$,} \\
\sigma_{y} & \text{otherwise.}
\end{cases}
\end{equation}
We define the energetic cost $E_{x}(\sigmavec)$ of flipping the spin value of $\sigmavec$ at $x$ by letting $E_{x}(\sigmavec) = H(\sigmavec^{x}) - H(\sigmavec)$, and denote by $E_{x}(\sigmavec)^{+}$ its positive part given by $E_{x}(\sigmavec)^{+} = \max\{E_{x}(\sigmavec), 0\}$. Note that  $E_{x}(\sigmavec)$ can be expressed in terms of the cavity field $\tilde{h}_{x}(\sigmavec)$, whose expression given by $\tilde{h}_{x}(\sigmavec)  = \sum_{y \in V} J_{x,y}\sigma_{y} + h_{x}$, in the form
\begin{equation}
E_{x}(\sigmavec) = 2 \tilde{h}_{x}(\sigmavec) \sigma_{x}.
\end{equation}

\begin{figure}
        \centering   
        \begin{subfigure}[b]{0.49\textwidth}
                \includegraphics[width=\textwidth]{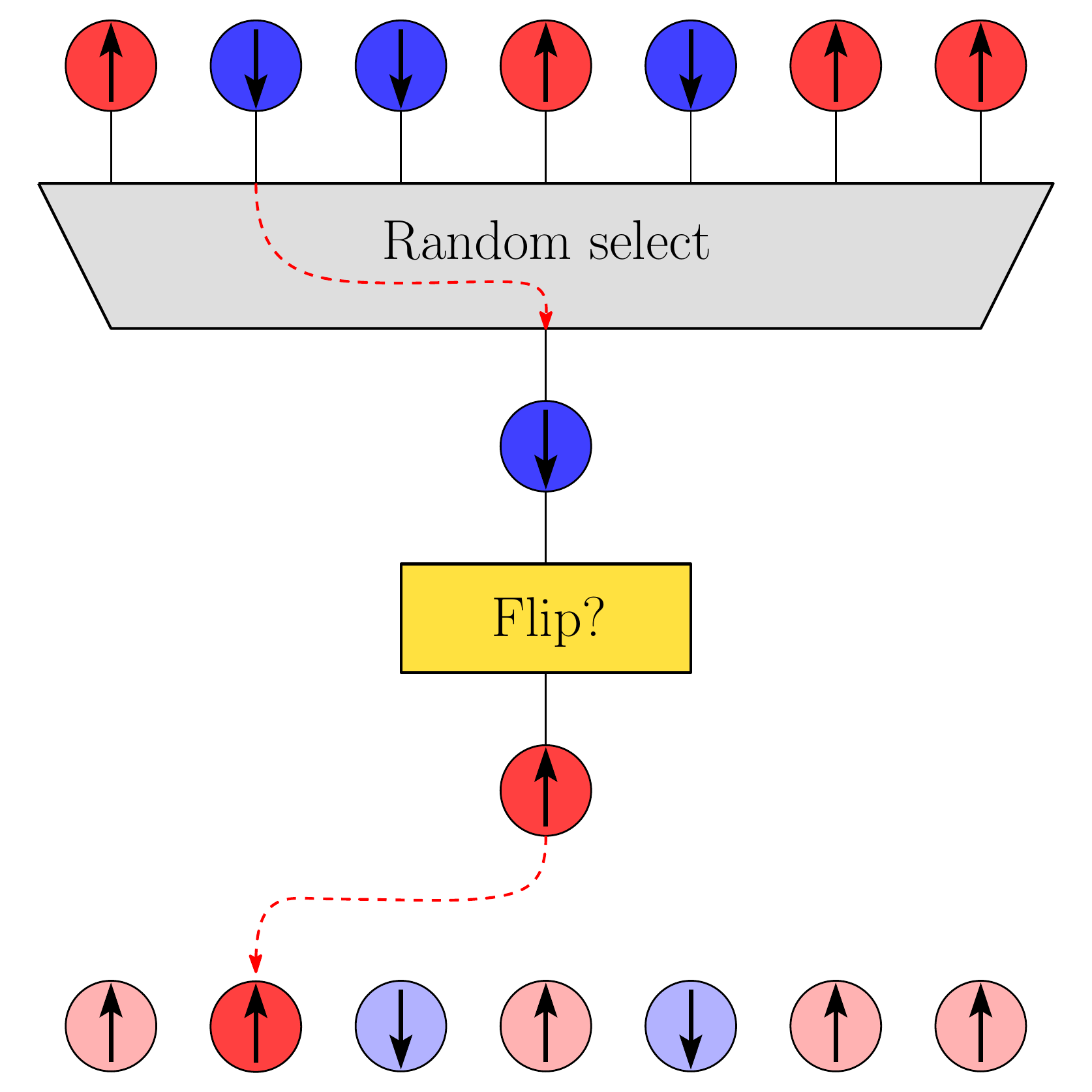}
                \caption{SA}\label{fig:SA}
        \end{subfigure}
        \begin{subfigure}[b]{0.49\textwidth}
                \includegraphics[width=\textwidth]{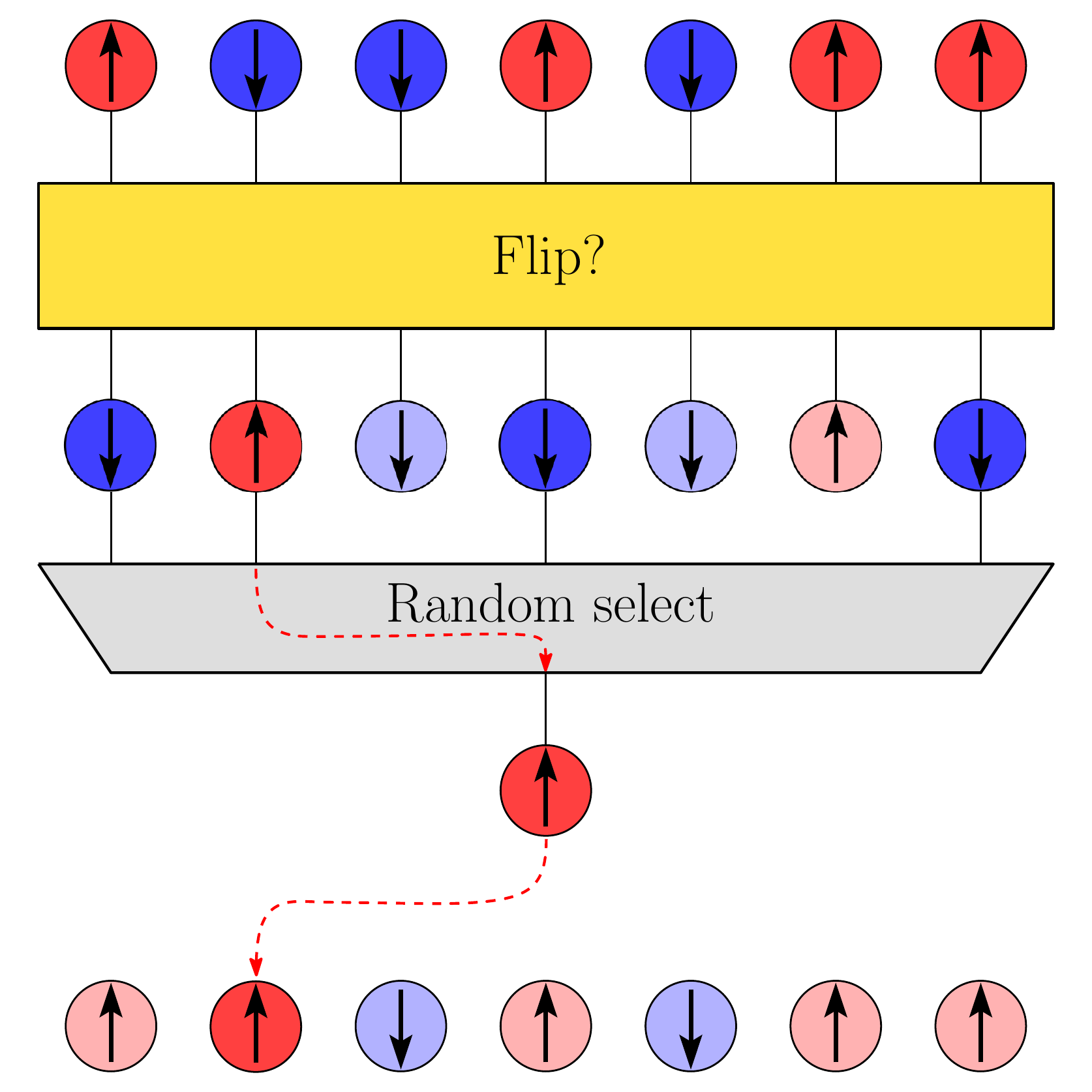}
                \caption{DA}\label{fig:DA}
        \end{subfigure}
    \caption{Comparison between the spin updates performed by SA and DA. We illustrate the spin update rules for these algorithms by decomposing them into two steps: the ``Random select'' phase and the ``Flip?'' phase. To help the reader visualize their operation, we represent the spin values $+1$ and $-1$ in red and blue, respectively, which, after being evaluated in the ``Flip?'' phase, are represented in transparent colors in case their new values are the same as in the original configuration. }
    \label{fig:SAvsDA}
\end{figure} 

The Digital Annealer's Algorithm works as follows. Given a state $X_{k-1} = \sigmavec$ at time $k-1$, we propose a parallel-trial where each vertex $x$ is assigned as eligible to flip its spin value according to the Metropolis criterion 
\begin{equation}\label{eq:metropoliscriterion}
\exp(-\beta_k E_x(\sigmavec)^+) = \min \{\exp(-\beta_k E_x(\sigmavec)), 1\}
\end{equation}
corresponding to temperature $1/{\beta_k}$. If the set $S$ of all eligible vertices contains at least one element, then a vertex $x$ is chosen uniformly at random from $S$, and we place $X_{k} = \sigmavec^{x}$; otherwise, nothing is done, and we consider $X_{k} = \sigmavec$. Note that, differently from \cite{MatsubaraDA, DA2017}, we are disregarding the escape mechanism so as to guarantee the Markov property holds. 
Although the SA and DA can flip at most one spin per update, they differ greatly in how they update spin configurations; see Figure \ref{fig:SAvsDA}. In Figure \ref{fig:SA}, given a spin configuration, we select one spin uniformly at random and decide whether we flip it or not according to Metropolis criterion (\ref{eq:metropoliscriterion}). In the end, the spin value is updated while the other spins remain unchanged. In Figure \ref{fig:DA}, differently from the SA, we first apply the ``Flip?'' phase, where we propose parallel spin-flips according to the Metropolis criterion (\ref{eq:metropoliscriterion}). Note that the only spins evaluated in the ``Random select'' phase are those represented in opaque colors, which correspond to the elements of the set $S$ of spins that accepted spin-flips, and we select only one of them, uniformly at random, to replace its corresponding value in the initial configuration.

According to practical simulations \cite{DA2019,TISCIE2022,IEEE2023,DA2017,Fujitsu2017}, it is expected that, by properly decreasing the temperature toward zero at each step, the updated configuration approaches a ground state. Therefore, a good approximation for a ground state can be achieved, provided we run the algorithm for a sufficiently large number of steps. In this paper, we prove Theorem \ref{DAannealing} to provide a mathematical foundation for these observations.

In our framework, the Digital Annealer's Algorithm transition matrix  $P_{\beta}^{\text{DA}}$ at inverse temperature $\beta \in (0,\infty)$ is defined by
\begin{equation}\label{eq:DAdef}
P_{\beta}^{\text{DA}}(\sigmavec, \tauvec) = 
\begin{cases} \displaystyle 
\sum_{\substack{S \subseteq V\\  S \ni x}} \frac{1}{|S|} \displaystyle \prod_{y \in S} e^{-\beta E_{y} (\sigmavec)^{+}} \displaystyle \prod_{y \in V \backslash S} (1-e^{-\beta E_{y} (\sigmavec) ^{+}}) & \text{if  $\tauvec = \sigmavec^{x}$ for some $x$ in $V$,} \\[20pt]

\displaystyle \prod_{y \in V} (1-e^{-\beta E_{y} (\sigmavec)^{+}})  & \text{if $\tauvec = \sigmavec$, and } \\[20pt]

0 & \text{otherwise}.
\end{cases}
\end{equation}
First, let us note that the spin-flip probability $P_{\beta}^{\text{DA}}(\sigmavec, \sigmavec^{x})$ can be decomposed in the form
\begin{equation}\label{eq:relation}
P_{\beta}^{\text{DA}}(\sigmavec, \sigmavec^{x}) = \frac{1}{|V|}  e^{-\beta E_{x}(\sigmavec)^{+}} 
+ \sum_{\substack{S \subseteq V\\  S \ni x}} \left(\frac{1}{|S|} - \frac{1}{|V|}\right) \displaystyle \prod_{y \in S} 
e^{-\beta E_{y}(\sigmavec)^{+}} \displaystyle \prod_{y \in V \backslash S} (1-e^{-\beta E_{y} (\sigmavec)^{+}}), 
\end{equation}
where the first term in the right-hand side of the equation (\ref{eq:relation}) coincides with the standard single site spin-flip Metropolis dynamics transition probability from $\sigmavec$ to $\sigmavec^{x}$.  Therefore, as discussed in \cite{DA2019}, the DA, in fact, provides us with a larger spin-flip rate compared to the Metropolis dynamics, which is generally attributed as a decisive property that can lead to greater effectiveness over SA.

The main result of this paper, Theorem \ref{DAannealing} stated below, provides us with a necessary and sufficient condition on the speed of convergence of the temperature $\beta_{k}^{-1}$ to zero so that the algorithm converges asymptotically to a ground state. Its proof is based on the techniques developed in \cite{Hajek88}, and it is presented in Section \ref{sec:proof}.

\begin{theorem}\label{DAannealing}
Let $(\beta_{k})_{k \geq 1}$ be a cooling schedule, that is, a nondecreasing sequence of positive real numbers such that  $\lim_{k \to \infty} \beta_{k}= + \infty$, and let  $(X_{k})_{k \geq 0}$ be the discrete-time inhomogeneous Markov chain satisfying
\begin{equation}\label{eq:Markov}
\mathbb{P}(X_{k} = \sigmavec_{k}| X_{k-1} = \sigmavec_{k-1}, \dots, X_{0} = \sigmavec_{0})
=\mathbb{P}(X_{k} = \sigmavec_{k}| X_{k-1} = \sigmavec_{k-1}) = P_{\beta_{k}}^{\text{DA}}(\sigmavec_{k-1}, \sigmavec_{k})
\end{equation}
for every positive integer $k$ and $\sigmavec_0, \dots, \sigmavec_k$ in $\Omega$. Then, there exists a constant $\gamma^{\ast}$ such that the limit
\begin{equation}\label{eq:limprob}
\lim_{k \to \infty}  \mathbb{P}(X_{k} \in \text{GS}) = 1
\end{equation}
holds if and only if
\begin{equation}\label{eq:suff_cond}
\sum_{k = 1}^{\infty} e^{-\beta_{k}\gamma^{\ast}} = +\infty.
\end{equation}
\end{theorem}

As we will precise more and prove in Section \ref{sec:proof}, the constant $\gamma^{\ast}$ from Theorem \ref{DAannealing} coincides with the depth of the deepest local minimum that is not a ground state. Note that, in particular, if we consider a logarithmic cooling schedule $(\beta_{k})_{k \geq 1}$ in the form
\begin{equation}
\beta_{k} = \frac{1}{\gamma} \log(k)
\end{equation}
for some positive parameter $\gamma$, then, equation (\ref{eq:limprob}) holds if and only if $\gamma \geq \gamma^{\ast}$. 
Let us note that since the necessary and sufficient condition (\ref{eq:suff_cond}) is the same as the one proven for the SA in \cite{Hajek88}, this result does not imply that the DA performs better than SA but only guarantees its convergence. In one possible attempt to provide rigorous statements about the superiority of one algorithm over another, it would be necessary to weaken the notion of convergence and derive results considering cooling schedules for finite-time simulations. Since each algorithm has its strengths and weaknesses and its performance is problem-dependent \cite{Mohseni2022}, such results are still challenging to derive and are out of the scope of the current state of our investigations.

\section{Stationary distributions}

The standard method used for proving the convergence of simulated annealing algorithms relies on obtaining the properties of weak and strong ergodicity for inhomogeneous Markov chains and on the analysis of their stationary distributions at fixed temperatures, see \cite{AK1989, b99, SM76, Senata81}. A mathematical foundation for the convergence of SA can be found in \cite{AK1989, b99}. Following the same idea, the convergence of an algorithm based on probabilistic cellular automata, namely the SCA, was obtained in \cite{SCA21}. In this section, we show that, differently from the Metropolis dynamics and the SCA, the stationary distribution for the DA transition probability does not necessarily assume a general formula or coincide with the Gibbs distribution. For that reason, in Section \ref{sec:proof} we adopt a different approach to prove convergence which does not depend upon the knowledge of the stationary distribution.

First, let us note that  $P_{\beta}^{\text{DA}}$ is an irreducible and aperiodic transition matrix. Thus, the discrete-time homogeneous Markov chain $(X_{k})_{k \geq 0}$ determined by $P_{\beta}^{\text{DA}}$ converges to its (unique) stationary distribution $\pi_{\beta}^{\text{DA}}$. In this section, we provide some examples where it is possible to determine the stationary distribution $\pi_{\beta}^{\text{DA}}$ of $P_{\beta}^{\text{DA}}$ and show that, for some particular cases, it differs from the Gibbs distribution $\pi_{\beta}^{\text{G}}$ defined by
\begin{equation}
\pi_{\beta}^{\text{G}}(\sigmavec) = \frac{e^{-\beta H(\sigmavec)}}{Z_{\beta}}
\end{equation}
for each configuration $\sigmavec$ in $\Omega$, where the normalizing factor $Z_{\beta} = \sum_{\tauvec \in \Omega} e^{-\beta H(\tauvec)}$ is known as the partition function.

\begin{example}\label{ex:2vertices}
Let us consider the simplest case where $V$ consists of only two vertices and the $H$ is the Hamiltonian of a ferromagnetic Ising model without external fields with pairwise interaction equal  $J \geq 0$. Let $\pi_{\beta}^{\text{DA}} = (\pi_{\beta}^{\text{DA}}({\uparrow \uparrow}), \pi_{\beta}^{\text{DA}}({\uparrow \downarrow}), \pi_{\beta}^{\text{DA}}({\downarrow \uparrow}), \pi_{\beta}^{\text{DA}}({\downarrow \downarrow}))$ be the stationary distribution of the transition matrix $P_{\beta}^{\text{DA}}$, which is expressed as the $4 \times 4$ matrix given by

\begin{equation} 
P_{\beta}^{\text{DA}} = 
\begin{pmatrix}  (1- e^{-2\beta J})^2 & e^{-2\beta J}(1-\frac12 e^{-2\beta J})  & e^{-2\beta J}(1-\frac12 e^{-2\beta J}) & 0 \\[5pt]
\frac{1}{2} & 0 & 0 &\frac{1}{2} \\[5pt]
\frac{1}{2} & 0 & 0 &\frac{1}{2} \\[5pt]
0 & e^{-2\beta J}(1-\frac12 e^{-2\beta J}) & e^{-2\beta J}(1-\frac12 e^{-2\beta J}) &  (1- e^{-2\beta J})^2 \\ 
\end{pmatrix}.
\end{equation}
It follows that $\pi_{\beta}^{\text{DA}}({\uparrow \uparrow}) =  \pi_{\beta}^{\text{DA}}({\downarrow \downarrow})$, $\pi_{\beta}^{\text{DA}}({\uparrow \downarrow}) = \pi_{\beta}^{\text{DA}}({\downarrow \uparrow})$, and 
\begin{equation}
\pi_{\beta}^{\text{DA}}({\uparrow \uparrow}) 
= \frac{e^{-\beta H(\uparrow \uparrow)}}{Z_{\beta} + 2 e^{-\beta J} (1 - e^{-2\beta J})} 
\end{equation}
and 
\begin{equation}
\pi_{\beta}^{\text{DA}}({\uparrow \downarrow}) 
= \frac{e^{-\beta H(\uparrow \downarrow)} +  e^{-\beta J} (1 - e^{-2\beta J})}{Z_{\beta} + 2 e^{-\beta J} (1 - e^{-2\beta J})}. 
\end{equation}
It is straightforward to show that $\pi_{\beta}^{\text{DA}}$ coincides with $\pi_{\beta}^{\text{G}}$ if and only if $J = 0$, moreover, $\pi_{\beta}^{\text{DA}}$ converges to the uniform distribution concentrated on the ground states $\uparrow \uparrow$ and $\downarrow \downarrow$ as $\beta$ tends to infinity.  $\QED$
\end{example}

 It follows from equation (\ref{eq:DAdef}) that the spin-flip probability $P_{\beta}^{\text{DA}}(\sigmavec, \sigmavec^{x})$ can also be expressed as
\begin{equation}\label{eq:relation2}
P_{\beta}^{\text{DA}}(\sigmavec, \sigmavec^{x}) = R(\sigmavec, \sigmavec^{x}) e^{-\beta E_{x}(\sigmavec)^{+}}, 
\end{equation}
where $R(\sigmavec, \sigmavec^{x})$ is defined by
\begin{equation}
R(\sigmavec, \sigmavec^{x}) = \sum_{\substack{S \subseteq V \setminus \{x\}}}\frac{1}{|S| + 1}  \displaystyle \prod_{y \in S} 
e^{-\beta E_{y}(\sigmavec)^{+}} \displaystyle \prod_{y \in (V \setminus \{x\}) \setminus S} (1-e^{-\beta E_{y} (\sigmavec)^{+}}). 
\end{equation}

\begin{prop}\label{prop:GDA}
Let $\pi_{\beta}^{\text{G}}$ be the Gibbs distribution on $\Omega$ at inverse temperature $\beta$. Then, it follows that 
\begin{equation}\label{eq:stationary}
(\pi_{\beta}^{\text{G}}P_{\beta}^{\text{DA}}) (\sigmavec) - \pi_{\beta}^{\text{G}}(\sigmavec) = \sum_{x \in V} e^{-\beta E_{x}(\sigmavec)^{+}} (R(\sigmavec^{x}, \sigmavec) - R(\sigmavec, \sigmavec^{x}))
\end{equation}
holds for every $\sigmavec$ in $\Omega$.
\end{prop}

\begin{proof}
Let $P_{\beta}^{\text{G}}$ denote the single site spin-flip Metropolis dynamics transition matrix at inverse temperature $\beta$ whose transition probability from $\sigmavec$ to $\sigmavec^{x}$ is given by $P_{\beta}^{\text{G}}(\sigmavec, \sigmavec^{x}) = \frac{1}{|V|} e^{-\beta E_{x}(\sigmavec)^{+}}$. Let us begin by noting that identities 
\begin{equation}
(\pi_{\beta}^{\text{G}}P_{\beta}^{\text{DA}}) (\sigmavec) = \pi_{\beta}^{\text{G}}(\sigmavec)  P_{\beta}^{\text{DA}}(\sigmavec, \sigmavec) + \sum_{x \in V} \pi_{\beta}^{\text{G}}(\sigmavec^{x})P_{\beta}^{\text{DA}} (\sigmavec^{x},\sigmavec)
\end{equation}
and
\begin{equation}
\pi_{\beta}^{\text{G}}(\sigmavec) = (\pi_{\beta}^{\text{G}}P_{\beta}^{\text{G}}) (\sigmavec) = \pi_{\beta}^{\text{G}}(\sigmavec)  P_{\beta}^{\text{G}}(\sigmavec, \sigmavec) + \sum_{x \in V} \pi_{\beta}^{\text{G}}(\sigmavec^{x})P_{\beta}^{\text{G}} (\sigmavec^{x},\sigmavec)
\end{equation}
imply
\begin{align}
(\pi_{\beta}^{\text{G}}P_{\beta}^{\text{DA}}) (\sigmavec)  - \pi_{\beta}^{\text{G}}(\sigmavec) =
 & \;\pi_{\beta}^{\text{G}}(\sigmavec) \left(P_{\beta}^{\text{DA}}(\sigmavec, \sigmavec) - P_{\beta}^{\text{G}}(\sigmavec, \sigmavec) \right) \\ 
 & + \sum_{x \in V} \pi_{\beta}^{\text{G}}(\sigmavec^{x}) \left(P_{\beta}^{\text{DA}}(\sigmavec^{x}, \sigmavec) - P_{\beta}^{\text{G}}(\sigmavec^{x}, \sigmavec) \right). \nonumber
 \end{align}
Now, by summing both sides of equation (\ref{eq:relation2}) over all vertices of $V$, we obtain
\begin{equation}\label{eq:GDA}
P_{\beta}^{\text{G}}(\sigmavec, \sigmavec) - P_{\beta}^{\text{DA}}(\sigmavec, \sigmavec) = 
\sum_{x \in V} \left(R(\sigmavec, \sigmavec^{x}) - \frac{1}{|V|}\right) e^{-\beta E_{x}(\sigmavec)^{+}}, 
\end{equation}
and, by exchanging the roles of $\sigmavec$ and $\sigmavec^{x}$ in equation (\ref{eq:relation2}), we have
\begin{equation}
P_{\beta}^{\text{DA}}(\sigmavec^{x}, \sigmavec) - P_{\beta}^{\text{G}}(\sigmavec^{x}, \sigmavec) =  
 \left(R(\sigmavec^{x}, \sigmavec) - \frac{1}{|V|}\right) e^{-\beta E_{x}(\sigmavec^{x})^{+}}.
\end{equation}
Therefore, since the Gibbs distribution satisfies the detailed balance equations for the Metropolis dynamics, i.e., $\pi_{\beta}^{\text{G}}(\sigmavec) e^{-\beta E_{x}(\sigmavec)^{+}} = \pi_{\beta}^{\text{G}}(\sigmavec^{x}) e^{-\beta E_{x}(\sigmavec^{x})^{+}}$, we derive the identity 
\begin{equation}
\pi_{\beta}^{\text{G}}P_{\beta}^{\text{DA}} (\sigmavec)  - \pi_{\beta}^{\text{G}}(\sigmavec)
 = \pi_{\beta}^{\text{G}}(\sigmavec) \left[P_{\beta}^{\text{DA}} (\sigmavec,\sigmavec) - P_{\beta}^{\text{G}} (\sigmavec,\sigmavec)
+ \sum_{x \in V} \left(R(\sigmavec^{x}, \sigmavec) - \frac{1}{|V|}\right) e^{-\beta E_{x}(\sigmavec)^{+}}\right],
\end{equation}
which combined with equation (\ref{eq:GDA}), implies equation (\ref{eq:stationary}).
\end{proof}

\begin{example}
Let us consider a Hamiltonian which is free of pairwise interactions, subject only to external fields, that is, the Hamiltonian $H$ is given in the form
\begin{equation}
H(\sigmavec)=-\sum_{x\in V}h_x \sigma_x.
\end{equation}
Note that, given a spin configuration $\sigmavec$, the identity $E_{y}(\sigmavec) = E_{y}(\sigmavec^{x})$ holds whenever $x$ and $y$ are distinct vertices of $V$, therefore, we have $R(\sigmavec, \sigmavec^{x}) = R(\sigmavec^{x}, \sigmavec)$ for each vertex $x$.  It follows from Proposition \ref{prop:GDA} that $\pi_{\beta}^{\text{G}}P_{\beta}^{\text{DA}} = \pi_{\beta}^{\text{G}}$, and consequently, $\pi_{\beta}^{\text{DA}} = \pi_{\beta}^{\text{G}}$. 
$\QED$
\end{example}

In the following proposition, we provide a  generalization of Example \ref{ex:2vertices}.
\begin{prop}
Let $H$ be the Hamiltonian of a ferromagnetic Ising model without external fields, which is given by
\begin{equation}
H(\sigmavec)=-\frac12\sum_{x,y\in V}J_{x,y}\sigma_x\sigma_y,
\end{equation}
where each pairwise interaction $J_{x,y}$ is a nonnegative real number.  It follows that $\pi_{\beta}^{\text{DA}} = \pi_{\beta}^{\text{G}}$ if and only if $J_{x,y}=0$ for all $x,y$ in $V$. 
\end{prop}
\begin{proof}
Suppose that there is a pair $v,w$ of vertices in $V$ such that $J_{v,w} > 0$. Let us consider the particular case where  $\sigmavec$ is a ground state of $H$,  which is a configuration whose spin values are all the same. Then, the identity 
\begin{equation}
E_{y}(\sigmavec^{x}) = E_{y}(\sigmavec) - 4 J_{x,y}  
\end{equation}
holds for any distinct points $x$ and $y$ in $V$. Corresponding to a fixed vertex $x$ in $V$, if we are given the families $(X_{i})_{i \in V}$ and $(Y_{i})_{i \in V}$  of independent Bernoulli random variables 
satisfying  $\mathbb{P}(X_{i} = 1) =  1 - \mathbb{P}(X_{i} = 0)= e^{-\beta E_{i}(\sigmavec)^{+}}$ and $\mathbb{P}(Y_{i} = 1) =  1 - \mathbb{P}(Y_{i} = 0)= e^{-\beta E_{i}(\sigmavec^{x})^{+}}$, it is straightforward to show that 
\begin{equation}\label{ex:RVineq}
R(\sigmavec, \sigmavec^{x})  = \mathbb{E}\left[\frac{1}{\sum_{i \in V} X_{i}} \Bigg| X_{x} = 1\right] 
\geq \mathbb{E}\left[\frac{1}{\sum_{i \in V} Y_{i}}\Bigg| Y_{x} = 1\right] = R(\sigmavec^{x}, \sigmavec).
\end{equation}
In particular, if we assume that $x = v$, the inequality from equation (\ref{ex:RVineq}) is strict, therefore, we conclude from Proposition \ref{prop:GDA} that $\pi_{\beta}^{\text{DA}} \neq \pi_{\beta}^{\text{G}}$. Conversely, suppose all pairwise interactions are null. In that case, it follows that $R(\sigmavec, \sigmavec^{x}) = R(\sigmavec^{x}, \sigmavec)$ for every configuration $\sigmavec$ and vertex $x$, thus, $\pi_{\beta}^{\text{G}}P_{\beta}^{\text{DA}} = \pi_{\beta}^{\text{G}}$.
\end{proof}

\section{Proof of Theorem \ref{DAannealing}}\label{sec:proof}

In this section, we derive Theorem \ref{DAannealing} as a particular case of \cite[Theorem 2]{Hajek88}. In order to do so, let us start by formulating our problem in a more general setting. 
Let us consider a nonempty finite state space $\Omega$, an energy function $H:\Omega \to \mathbb{R}$, and a system $\mathcal{N} = \{N(\sigmavec): \sigmavec \in \Omega\}$ of neighbors of states. Recall that we are particularly interested in the set of ground states (global minima) of $H$ defined as 
\begin{equation}
    \text{GS} = \{\sigmavec \in \Omega: H(\sigmavec) = \min_{\tauvec \in \Omega} H(\tauvec)\}. 
\end{equation}
In this setting, the property of irreducibility can be defined as follows. 

\begin{definition}[Irreducibility]
The pair $(\Omega,\mathcal{N})$ is said to be irreducible if for any distinct states $\sigmavec$ and $\tauvec$ there is a sequence $\sigmavec_{0} = \sigmavec$, $\sigmavec_{1}$, $\dots$, $\sigmavec_{n} = \tauvec$ in $\Omega$, where $n \geq 1$, such that
$\sigmavec_{k} \in N(\sigmavec_{k-1})$ whenever $1 \leq k \leq n$. 
\end{definition}

In addition to the property of irreducibility, it will be necessary to assume weak reversibility. Before introducing such a notion, let us define the idea of reachability. For any pair $\sigmavec,\tauvec$ of points in $\Omega$, $\tauvec$ is said to be reachable from $\sigmavec$ at height $E$ if either  $\sigmavec = \tauvec$ and $H(\sigmavec) \leq E$, or $\sigmavec \neq \tauvec$ and there exists a sequence of states $\sigmavec_{0} = \sigmavec$, $\sigmavec_{1}$, $\dots$, $\sigmavec_{n} = \tauvec$, where $n \geq 1$, such that  $\max_{0 \leq i \leq n}H(\sigmavec_{i}) \leq E$ and 
$\sigmavec_{k} \in N(\sigmavec_{k-1})$ for each $k$ such that $1 \leq k \leq n$. 

\begin{definition}[Weak reversibility]
The triple $(\Omega, H,\mathcal{N})$ is said to be weakly reversible if for any states $\sigmavec$ and $\tauvec$ and any real number $E$, $\sigmavec$ is reachable from $\tauvec$ at height $E$ if and only if $\tauvec$ is reachable from $\sigmavec$ at height $E$. 
\end{definition}

From now on, let us assume the properties of irreducibility and weak reversibility. 
In order to state Theorem \ref{DAannealing2}, it is necessary to introduce the notions of a cup and local minima. A subset $C$ of $\Omega$ is called a cup if there exists some real number $E$ such that for any $\sigmavec \in C$, the set $C$ can be expressed as the set of all states that are reachable from $\sigmavec$ at height $E$.
Then, given a cup $C$, we define its boundary by  
\begin{equation}
\partial C = \{\tauvec \notin C: \text{$\tauvec \in N(\sigmavec)$ for some $\sigmavec$ in $C$}\}, 
\end{equation}
its bottom by 
\begin{equation}
B(C) = \{\sigmavec \in C: H(\sigmavec) = \min_{\tauvec \in C} H(\tauvec)\}, 
\end{equation}
and its depth by 
\begin{equation}
d(C) = \min_{\sigmavec \in \partial C}H(\sigmavec) - \min_{\sigmavec \in C} H(\sigmavec). 
\end{equation}
Furthermore, we say that $\sigmavec$ is a local minimum if there is no $\tauvec$ satisfying $H(\tauvec) < H(\sigmavec)$ which is reachable from $\sigmavec$ at height $H(\sigmavec)$. So, the depth of a local minimum $\sigmavec$ which is not a ground state is defined as the smallest positive real number $E$ so that there exists $\tauvec$ with $H(\tauvec) < H(\sigmavec)$ that is reachable from $\sigmavec$ at height $H(\sigmavec) + E$; otherwise, in case $\sigmavec$ is a ground state, its depth is defined as infinity. It is straightforward to show that every local minimum of depth $d$ is contained in the bottom of a cup of depth $d$.

 Although the motivation of the problem proposed by B. Hajek in \cite{Hajek88} was formulated with the Metropolis dynamics in mind, his main result was proven in a rather more general setting. Hajek considered a continuous-time process $(W_t)_{t \in [0,\infty)}$ defined by letting $W_t = X_k$ whenever $U_k \leq t < U_{k+1}$, derived from a discrete-time homogeneous Markov process $(U_{k}, X_{k})_{k \geq 0}$ on the state space $[0,\infty) \times \Omega$ whose transition probabilities assume the form 
\begin{equation}\label{eq:Markov2}
\mathbb{P}(U_{k+1} \geq u, X_{k+1} = \tauvec| \text{$(U_{i}, X_{i}) = (u_{i}, \sigmavec_{i})$ for all $0 \leq i \leq k$}) = \int_{u}^{\infty}Q(\sigmavec_{k},\tauvec,t,\lambda_{t})\Phi(dt, u_{k},\sigmavec_{k}).
\end{equation}
In the equation above, we suppose that $(\lambda_t)_{t \in [0,\infty)}$ is a nonincreasing family of numbers in the interval $(0,1)$ such that $\lim_{t \to \infty}\lambda_{t} = 0$.
Moreover, we assume the existence of positive constants $a, D, c_1$, and $c_2$ such that 
\begin{itemize}
    \item $\mathbb{E}[\min\{k: U_k \geq t + a\}| U_{0} = t, X_{0} = \sigmavec] \leq \frac{1}{a}$ holds for all $t \geq 0$ and $\sigmavec \in \Omega$.

\item $\Phi(\cdot, s,\sigmavec)$ is a probability distribution function such that 
\begin{equation}
    \Phi(s^{-},s,\sigmavec) = 0
\end{equation}
and 
\begin{equation}
    \int_{0}^{\infty} (t-s)\Phi(dt,s,\sigmavec) \leq D.
\end{equation}

\item $Q(\cdot,\cdot,t,\lambda)$ is an stochastic matrix such that the conditions
    \begin{equation}
        \text{$c_1 \lambda^{[H(\tauvec) - H(\sigmavec)]^{+}} \leq Q(\sigmavec,\tauvec,t,\lambda) \leq c_2 \lambda^{[H(\tauvec) - H(\sigmavec)]^{+}}$ for $\tauvec \in N(\sigmavec)$}
    \end{equation}
and
    \begin{equation}
        \text{$Q(\sigmavec,\tauvec,t,\lambda) = 0$ for $\tauvec \notin N(\sigmavec) \cup \{\sigmavec\}$}
    \end{equation}
    hold whenever $t \geq 0$ and $0 \leq \lambda \leq 1$.
\end{itemize}

\begin{theorem}[\cite{Hajek88}]\label{DAannealing2}
Let us assume that $(\Omega,H,\mathcal{N})$ is irreducible and weakly reversible, and $(W_{t})_{t \in [0,\infty)}$ is the continuous-time process defined as above. Then, the following conditions hold.

\begin{enumerate}[label=(\alph*),ref=\alph*]
\item Every state $\sigmavec$ which is not a local minimum satisfies
\begin{equation}
\lim_{t \to \infty}  \mathbb{P}(W_{t} = \sigmavec) = 0.
\end{equation}

\item Given a cup $C$ of depth $d$ whose elements of its bottom are local minima of depth $d$, then
\begin{equation}
\lim_{t \to \infty}  \mathbb{P}(W_{t} \in B(C)) = 0
\end{equation}
if and only if
\begin{equation}
\int_{0}^{\infty} \lambda_{t}^{d} dt  = +\infty.
\end{equation}

\item\label{item:annealing} If we define  $\gamma^{\ast}$ as the depth of the second deepest local minimum, then the limit
\begin{equation}\label{eq:limprob2}
\lim_{t \to \infty}  \mathbb{P}(W_{t} \in \text{GS}) = 1
\end{equation}
holds if and only if
\begin{equation}
\int_{0}^{\infty} \lambda_{t}^{\gamma^{*}} dt = +\infty.
\end{equation}
\end{enumerate}
\end{theorem}

\begin{proof}[Proof of Theorem \ref{DAannealing} assuming Theorem \ref{DAannealing2}]
Let us show that the setting assumed in Theorem \ref{DAannealing} fits as a particular case of Theorem \ref{DAannealing2}. In the following, let us consider the triple $(\Omega, H, \mathcal{N})$, where $\Omega = \{-1,+1\}^{V}$ is the set of all Ising spin configurations on the vertex set $V$ of a graph, $H$ is the Hamiltonian given by equation (\ref{eq:Hamiltonian}), and $\mathcal{N} = \{N(\sigmavec): \sigmavec \in \Omega\}$ is the system of neighbors of spin configurations defined as follows. Given a configuration $\sigmavec$ in $\Omega$, let $N(\sigmavec)$ be the set of neighbors of $\sigmavec$ defined by 
\begin{equation}
N(\sigmavec) = \{\sigmavec^{x}: x \in V\}.
\end{equation}
It is straightforward to show that the triple $(\Omega, H, \mathcal{N})$ is irreducible and weakly reversible.

Let us consider the discrete-time homogeneous Markov chain $(U_{k}, X_{k})_{k \geq 0}$ on the state space $[0,\infty) \times \Omega$ whose transition probabilities assume the form (\ref{eq:Markov2}), with the transition matrix $Q(\cdot, \cdot, t, \lambda)$, the probability distribution function $\Phi(\cdot ,s,\sigmavec)$ and the family of real numbers $(\lambda_{t})_{t \in [0,\infty)}$ chosen by letting $Q(\cdot, \cdot, t, \lambda) = P_{\log \lambda^{-1}}^{\text{DA}}$, $\Phi(t ,s,\sigmavec) = \indic_{\{t \geq s+1\}}$, and  
\[ \lambda_{t} =
\begin{cases}
    \exp(-\beta_{\lfloor t \rfloor}) &\text{if $t \geq 1$,}\\
    \exp(-\beta_{1}) &\text{otherwise;}
\end{cases}
\] 
where $\lfloor t \rfloor$ stands for the greatest integer less than or equal to $t$. Note that equations (\ref{eq:DAdef}) and  (\ref{eq:relation}) imply that 
\begin{equation}
\frac{1}{|V|}  \lambda^{[H(\tauvec) - H(\sigmavec)]^{+}} \leq Q(\sigmavec, \tauvec, t, \lambda) \leq \lambda^{[H(\tauvec) - H(\sigmavec)]^{+}}
\end{equation}
holds for each $\tauvec$ in $N(\sigmavec)$, and $Q(\sigmavec, \tauvec, t, \lambda) = 0$ whenever $\tauvec$ does not belong to $N(\sigmavec) \cup \{\sigmavec\}$. Furthermore, equation (\ref{eq:Markov2}) is expressed as 
\begin{equation}
\mathbb{P}(U_{k+1} \geq u, X_{k+1} = \tauvec | \text{$(U_{i}, X_{i}) = (u_{i}, \sigmavec_{i})$ for all $0 \leq i \leq k$}) =   P_{\beta_{\lfloor u_{k} + 1\rfloor}}^{\text{DA}}(\sigmavec_{k},\tauvec) \indic_{\{u_{k}+1 \geq u\}},
\end{equation}
which implies that 
\begin{equation}
\mathbb{P}(U_{k+1} = u_{k} + 1, X_{k+1} = \tauvec |  \text{$(U_{i}, X_{i}) = (u_{i}, \sigmavec_{i})$ for all $0 \leq i \leq k$}) =   P_{\beta_{\lfloor u_{k} + 1\rfloor}}^{\text{DA}}(\sigmavec_{k},\tauvec).
\end{equation}
In the particular case where $U_0 = 0$, it is straightforward to show that $\mathbb{P}(U_k=k) = 1$ for each $k \geq 0$ and the one-step transition probabilities of $(X_k)_{k \geq 0}$ take the form (\ref{eq:Markov}). Since the hypotheses of Theorem \ref{DAannealing2} are applicable to this case, then, by using the continuous-time stochastic process $(W_t)_{t \in [0, \infty)}$ evaluated on integer times and Theorem \ref{DAannealing2}(\ref{item:annealing}), the result follows.
\end{proof}

\section{Concluding remarks}
First, let us recall that the general algorithm introduced in \cite{DA2019} relies on the implementation of the so-called dynamic offset in order to encourage the system to flip more spins and prevent it from getting trapped in a local minimum. Thus, in order to prove statements regarding the convergence to the ground states for such an algorithm, one should rely on a different approach since the Markov property would no longer be valid.

On the other hand, note that in the proof of Theorem \ref{DAannealing2} the assumption that $H$ is given as in equation (\ref{eq:Hamiltonian}) was not necessary, therefore,  we do not have to restrict ourselves only to QUBO problems and Theorem \ref{DAannealing} can also be applied to search for ground states of Hamiltonians other than those written in this form. Furthermore, recall that, as discussed in Section \ref{sec:DA}, we dealt with results regarding the asymptotic convergence to the ground states, and, in a realistic scenario, especially when dealing with problems that involve thousands of variables, one does not have infinite time to perform a simulation. Relying on ideas from Freidlin and Wentzel \cite{freidlin2012}, O. Catoni \cite{Catoni1991,Catoni1992,Cot1998} obtained rigorous results regarding finite time simulation, large deviation principles, and optimal cooling schedules for the Metropolis algorithm and Markov chains with rare transitions. Investigations on such topics specifically for the DA and other algorithms (such as those discussed in \cite{SCA21}), including rigorous results concerning the effectiveness of one algorithm over the other, are still in progress.

Another topic that is also worth investigating is the behavior of the corresponding time-homogeneous Markov chain at a fixed temperature in order to provide a complete characterization of its stationary distribution and obtain estimates of its mixing time. As the reader can verify, providing an upper bound for the mixing time at high temperatures such as in \cite{SCA21,TISCIE2022}, where it was possible to show that the mixing times for certain probabilistic cellular automata is upper bounded by a quantity proportional to $\log |V|$, is not straightforward. Hence, there are still interesting directions to be explored.

\section*{Declarations}
This work was supported by JST CREST Grant Number PJ22180021, Japan. Data sharing is not applicable to this article as no datasets were generated or analyzed during the current study. The authors have no relevant financial or non-financial interests to disclose.

\section*{Acknowledgements}
We are grateful to the following members for their continual encouragement and stimulating discussions: 
Masato Motomura and Kazushi Kawamura from the Tokyo Institute of Technology; 
Hiroshi Teramoto from Kansai University; 
Masamitsu Aoki from the Graduate School of 
Mathematics at Hokkaido University. The authors would like to thank 
Thiago Raszeja, Olena Karpel, Dominik Kwietniak, Wioletta Ruszel, Cristian Spitoni, Ross Kang, and Lars Fritz for the hospitality and for providing nice opportunities for discussion during Bruno Hideki Fukushima Kimura's visiting period at the
AGH University of Science and Technology and Utrecht University.

\bibliographystyle{plain}
\bibliography{bibliography}

\end{document}